\documentclass[12pt]{amsart}

\usepackage{amsfonts}
\usepackage{tipa}
\usepackage{amssymb}
\usepackage{mathrsfs}
\usepackage{amsmath}
\usepackage{txfonts}
\theoremstyle{plain}
\newtheorem{Thm}{Theorem}

\newtheorem{Cor}[Thm]{Corollary}
\newtheorem{Pro}[Thm]{Proposition}
\newtheorem{Lem}[Thm]{Lemma}

\newtheorem{Rk}[Thm]{Remark}

\newtheorem{Claim}{Claim}

\begin{document}

\title[Symmetry for Fractional Laplacian]{Radial symmetry results for fractional Laplacian systems}

\author[B. Liu]{Baiyu Liu}
\author[L. Ma]{Li Ma}

\address[B. Liu]{School of Mathematics and Physics\\
  University of Science and Technology Beijing \\
  30 Xueyuan Road, Haidian District
  Beijing, 100083\\
  P.R. China}
\email{liuby@ustb.edu.cn}

\address[L. Ma]{Department of mathematics \\
Henan Normal university \\
Xinxiang, 453007 \\
China}

\email{nuslma@gmail.com}

\thanks{$^*$ Project supported by  the National Natural Science Foundation of China (No. 11201025, No. 11271111), the Doctoral Program Foundation of the Ministry of Education of China (No. 20090002110019) and the Fundamental Research Funds for the Central Universities FRF-TP-15-037A3}

\begin{abstract}
In this paper, we generalize the direct method of moving planes for the fractional Laplacian to the system case. Considering a coupled nonlinear system with fractional Laplacian, we first establish a decay at infinity principle and a narrow region principle. Using these principles, we obtain two radial symmetry results for the decaying solutions of the fractional Laplacian systems. Our method can be applied to fractional Schr\"odinger systems and fractional H\'enon systems.
\end{abstract}

\maketitle

\section{Introduction}

In this paper, we are concerned with the symmetry properties of the following system
\begin{equation}\label{sys:fes}
\left\{
\begin{array}{ll}
(-\Delta)^{\frac{\alpha}{2}}u(x)=f(u,v), & x\in \mathbb{R}^n,
\\
(-\Delta)^{\frac{\alpha}{2}}v(x)=g(u,v), & x\in \mathbb{R}^n,
\\
u(x)>0, v(x)>0, & x\in \mathbb{R}^n.
\end{array}
\right.
\end{equation}
where $\alpha$ is any real number between $0$ and $2$.
Here,
$(-\Delta)^{\frac{\alpha}{2}}$ is the fractional Laplacian operator, which is a nonlocal pseudo-differential operator assuming the form
\begin{equation}
\label{def:fralap}
(-\Delta)^{\frac{\alpha}{2}}u(x)=C_{n,\alpha}\lim_{\epsilon\to 0}\int_{\mathbb{R}^n\backslash B_\epsilon(x)}\frac{u(x)-u(y)}{|x-y|^{n+\alpha}}dy.
\end{equation}
$C_{n,\alpha}$ is a normalization positive constant.
Let
$$
L_{\alpha}=\{u:\mathbb{R}^n\to \mathbb{R}|\int_{\mathbb{R}^n}\frac{|u(x)|}{1+|x|^{n+\alpha}}dx<+\infty \}$$ (see \cite{Sil2007}).
For
$u,v\in L_\alpha\cap C^{1,1}_{loc}(\mathbb{R}^n)$, the integrals on the left hand side of (\ref{sys:fes}) are well defined.

Recently, equations and systems involving the fractional Laplacian 
have been extensively studied, which have been used to model diverse physical phenomenas, such as anomalous diffusion, quasi-geostrophic flow, boundary control problems, and posudo-relativistic boson stars (see \cite{BG1990, CV2010, Con2006, FJL} and the references therein).

In this paper, we are interested in the radial symmetry result of system (\ref{sys:fes}).
When $\alpha=2$, the symmetry properties for solutions of system (\ref{sys:fes}), which is the classical semi-linear elliptic system, have been widely studied in a number of literatures (see \cite{FF1994,BS2000, ML2010,GNN1979,CL2010,LM2012,LM2013} and the references therein). A powerful method to obtain the symmetry of such kind of equations (and systems) is the method of moving planes, introduced by Alexandroff and improved by 
Berestycki and Nirenberg \cite{BN1991}.

As we see the fractional Laplacian (\ref{def:fralap}) is nonlocal, i.e. it does not act by pointwise differentiation but a global integral with respect to a singular kernel, that causes the main difficulty in studying problems involving it. Caffarelli and Silvestre \cite{CS2007} introduced an extension method (CS extension) to overcome this nonlocal difficulty. Their idea is to localize the fractional Laplacian by constructing a Dirichlet to Neumann operator of a degenerate elliptic equation. Using the CS extension, Brandle et al. \cite{BCPS2013} derived the non-existence of solution for 
\begin{equation}
\label{eq:fl}
(-\Delta)^{\alpha/2} u=u^p(x),\quad x\in \mathbb{R}^n,
\end{equation}
when $1\leq \alpha<2$ and $1<p<(n+\alpha)/(n-\alpha)$. In the same spirit, the symmetry among other properties of solutions for a wide kinds of fractional Laplacian equations have been obtained in \cite{CZ2014,QX2015,FW2012,CT2010}.

Another method to deal with the nonlocal difficulty is considering the corresponding equivalent integral equation of the fractional Laplacian equation. Using the integral form moving plane (or sphere) method \cite{CLO2006, CLO2005dcds, CLO2005cpde, Li2004, LM2008, MC06, JL2006, Yu2013} to treat the integral equation, one can obtain the symmetry properties of the fractional Laplacian equation. An alternative is developing new kinds of Maximum Principle and applying them to the nonlocal problem. We mention here the earlier work by Jarohs and Weth \cite{JW2016}, Felmer et al. \cite{FQT2012}, Felmer and Wang \cite{FW2014CCM} and Zhuo et al. \cite{ZCCY2016}. In \cite{FW2014CCM}, by using a form of Maximum Principle for domains with small measure together with the moving plane method, assuming that
\begin{enumerate}
	\item[(F1)]  $f$ is locally Lipschitz;
	\item[(F2)] there exist $s_0,\gamma,C>0$ such that
	$$
	\frac{f(v)-f(u)}{v-u}\leq (v+u)^{\gamma}, \quad \textrm{for all} \ 0<u<v<s_0,
	$$
\end{enumerate}
Felmer and Wang obtain the radial symmetry for positive decay solution of equation
\begin{equation*}
(-\Delta)^{\alpha/2}u=f(u),\ x\in \mathbb{R}^n,
\end{equation*}
provided that $u(x)=O(|x|^{-m})$, as $|x|\to \infty$ and 
$m>\max\{\frac{\alpha}{2\gamma},\frac{n}{\gamma+2}\}$.
In \cite{ZCCY2016}, Zhuo et al. 
considered the system 
\begin{equation}\label{sys:zhuo}
(-\Delta)^{\alpha/2}u_i(x)=f_i(u_1(x),\dots, u_m(x)),\ x\in \mathbb{R}^n,\ i=1,2,\dots, m,
\end{equation}
under assumption:
(f) the nonlinear terms $f_i$ are real-valued, non-negative, continuous, homogeneous of degree $1<\gamma<(n+\alpha)/(n-\alpha)$ and nondecreasing with respect to the variables $u_1,\dots, u_m$.
By establishing the equivalence between (\ref{sys:zhuo}) and an integral system and using properties of the integral system \cite{CL2009}, the authors of \cite{ZCCY2016} obtained the symmetry result and hence the non-existence of positive solutions for (\ref{sys:zhuo}). We remark that our results Theorem \ref{thm:main} and \ref{thm:rev} assume weaker monotonicity for the nonlinear terms and we do not assume the the nonlinear terms to be homogeneous. Moreover, our results can be applied to the supercritical case.

Recently, Chen, Li and Li \cite{CLL14} developed a direct method of moving planes to study the fractional Laplacian equation, which worked directly on the nonlocal operator. By showing radially symmetry for solution of (\ref{eq:fl}) with the direct method of moving plane, Chen et al.\cite{CLL14} proved that (\ref{eq:fl}) has no positive in $L_\alpha\cap C_{loc}^{1,1}(\mathbb{R}^n)$ in the subcritical case and the solution is radially symmetric in the critical case for all $\alpha\in (0,2)$.
The motivation of this paper is to generalize the direct method of moving planes \cite{CLL14} to the system case.

Here is the main result of this paper.
\begin{Thm}\label{thm:main}
	Let $(u,v)\in \left(L_{\alpha}\cap C^{1,1}_{loc}(\mathbb{R}^n)\right)^2$ be a positive solution of system (\ref{sys:fes}) with $f,g\in C^1([0,+\infty)\times [0,+\infty),\mathbb{R})$ and $\alpha\in (0,2)$.
	We suppose that
	\begin{eqnarray*}
	(i) & & u(x)\lesssim \frac{1}{|x|^a}, \quad v(x)\lesssim \frac{1}{|x|^b}, \quad \textrm{as}\ |x|\to \infty;\\
	(ii) & & \frac{\partial f}{\partial u}\lesssim u^{p-1}v^{q},\quad \frac{\partial g}{\partial v}\lesssim u^{r}v^{s-1}, \quad \textrm{as}\ (u,v)\to (0^+,0^+);\\
	(iii) & & \frac{\partial f}{\partial v}\lesssim u^{p}v^{q-1},\quad \frac{\partial g}{\partial u}\lesssim u^{r-1}v^{s},  \quad \textrm{as}\ (u,v)\to (0^+,0^+);\\
	(iv) & & \frac{\partial f}{\partial v}(u,v)>0,\quad \frac{\partial g}{\partial u}>0,\quad  \forall (u,v)\in  \mathbb{R}^+\times \mathbb{R}^+,
	\end{eqnarray*}
	where $a,b>0$; $p,q,r,s\geq 1$, and satisfying
	$$
	(v)\quad \alpha<\min\{ap+bq-a, ap+bq-b, ar+bs-a, ar+bs-b\}.
	$$
	Then there exists a point $x_0\in \mathbb{R}^n$ such that
	$$
	u(x)=u(|x-x_0|),\quad v(x)=v(|x-x_0|).
	$$
\end{Thm}

\begin{Rk}
	If $a(x)$ and $b(x)$ are two functions, we use
	$a(x)\lesssim b(x)$ as $|x|\to 0^+$ to denote the statement that for
	$x>0$ sufficiently close to $0$, $a(x)\leq b(x)$.
\end{Rk}

An important consequence of our result is the radial symmetry for the fractional Schr\"odinger system as follows.
\begin{Cor}\label{col:ex1}
	Let $(u,v)\in \left(L_{\alpha}\cap C^{1,1}_{loc}(\mathbb{R}^n)\right)^2$ be a positive solution of system
	\begin{equation*}
	\left\{
	\begin{array}{ll}
	(-\Delta)^{\frac{\alpha}{2}}u(x)=\sum_{i=1}^Nu^{p_i}v^{q_i}, & x\in \mathbb{R}^n,
	\\
	(-\Delta)^{\frac{\alpha}{2}}v(x)=\sum_{i=1}^Nu^{r_i}v^{s_i}, & x\in \mathbb{R}^n,
	\\
	u(x)>0, v(x)>0, & x\in \mathbb{R}^n,
	\end{array}
	\right.
	\end{equation*}
	where $\alpha\in (0,2)$ and $q_i,p_i,r_i,s_i\geq 1$, $i=1,2,\dots,N$.
	Let $p=\min_{1\leq i\leq N}\{p_i\}$, $q=\min_{1\leq i\leq N}\{q_i\}$, $r=\min_{1\leq i\leq N}\{r_i\}$ and $s=\min_{1\leq i\leq N}\{s_i\}$.
	Suppose that:\\
	$(i)$ $u(x)\lesssim |x|^{-a}, v(x)\lesssim |x|^{-b}$ $(a,b>0)$, as $|x|\to \infty$;\\
	$(v)$ $ \alpha<\min\{ap+bq-a, ap+bq-b, ar+bs-a, ar+bs-b\}$.
	
	Then there exists a point $x_0\in \mathbb{R}^n$ such that
	$$
	u(x)=u(|x-x_0|),\quad v(x)=v(|x-x_0|).
	$$
\end{Cor}

\begin{Thm}\label{thm:rev}
	Let $(u,v)\in \left(L_{\alpha}\cap C^{1,1}_{loc}(\mathbb{R}^n)\right)^2$ be a positive solution of system (\ref{sys:fes}) with $f,g\in C^1([0,+\infty)\times [0,+\infty),\mathbb{R})$ and $\alpha\in (0,2)$.
	We suppose that
	\begin{eqnarray*}
		(i) & & u(x)\lesssim \frac{1}{|x|^a}, \quad v(x)\lesssim \frac{1}{|x|^b}, \quad \textrm{as}\ |x|\to \infty;\\
		(ii') & & \frac{\partial f}{\partial u}\lesssim u^{p-1},\quad \frac{\partial g}{\partial v}\lesssim v^{s-1}, \quad \textrm{as}\ (u,v)\to (0^+,0^+);\\
		(iii') & & \frac{\partial f}{\partial v}\lesssim v^{q-1},\quad \frac{\partial g}{\partial u}\lesssim u^{r-1},  \quad \textrm{as}\ (u,v)\to (0^+,0^+);\\
		(iv) & & \frac{\partial f}{\partial v}(u,v)>0,\quad \frac{\partial g}{\partial u}>0,\quad  \forall (u,v)\in  \mathbb{R}^+\times \mathbb{R}^+,
	\end{eqnarray*}
	where $a,b>0$; $p,q,r,s>1$, and satisfies
	$$
	(v')\quad \alpha<\min\{a(p-1), b(q-1), a(r-1), b(s-1)\}.
	$$
	Then there exists a point $x_0\in \mathbb{R}^n$ such that
	$$
	u(x)=u(|x-x_0|),\quad v(x)=v(|x-x_0|).
	$$
\end{Thm}

\begin{Rk}
	Theorem \ref{thm:rev} is a fractional Laplacian version of Theorem 1 in our early work \cite{ML2010}, which deals with the classic elliptic system. Also, it is a generalization of Theorem 1.2 in \cite{FW2014CCM} to the system case.
\end{Rk}

As an application, we apply Theorem \ref{thm:rev} to the following system and obtain the radial symmetry property for the decay solution.
\begin{Cor}\label{col:ex2}
	Let $(u,v)\in \left(L_{\alpha}\cap C^{1,1}_{loc}(\mathbb{R}^n)\right)^2$ be a positive solution of system
	\begin{equation*}
	\left\{
	\begin{array}{ll}
	(-\Delta)^{\frac{\alpha}{2}}u(x)=v^\frac{n+\alpha}{n-\alpha}-\kappa_1 u^p, & x\in \mathbb{R}^n,
	\\
	(-\Delta)^{\frac{\alpha}{2}}v(x)=u^\frac{n+\alpha}{n-\alpha}-\kappa_2 v^s, & x\in \mathbb{R}^n,
	\\
	u(x)>0, v(x)>0, & x\in \mathbb{R}^n,
	\end{array}
	\right.
	\end{equation*}
	where $\alpha\in (0,2)$, $p,s>1$ and $\kappa_1,\kappa_2>0$.
	If $u(x)\lesssim |x|^{-(n-\alpha)}$ and  $v(x)\lesssim |x|^{-(n-\alpha)}$ as $|x|\to \infty$,
	then there exists a point $x_0\in \mathbb{R}^n$ such that
	$$
	u(x)=u(|x-x_0|),\quad v(x)=v(|x-x_0|).
	$$
\end{Cor}

In the end, using the direct method of moving plane for systems, we find the symmetry for solutions of fractional H\'enon system. Although similar result is known (see \cite{DZ2015}), we include it here as an application of our method.
\begin{Thm}\label{thm:henon}
	Let $(u,v)\in \left(L_{\alpha}\cap C^{1,1}_{loc}(\mathbb{R}^n)\right)^2$ be a positive solution of system
	\begin{equation}\label{sys:fes3}
	\left\{
	\begin{array}{ll}
	(-\Delta)^{\frac{\alpha}{2}}u(x)=|x|^{\sigma_1}v^q, & x\in \mathbb{R}^n,
	\\
	(-\Delta)^{\frac{\alpha}{2}}v(x)=|x|^{\sigma_2}u^r, & x\in \mathbb{R}^n,
	\\
	u(x)>0, v(x)>0, & x\in \mathbb{R}^n,
	\end{array}
	\right.
	\end{equation}
	where $\alpha\in (0,2)$, $q,r>\frac{n+\alpha}{n-\alpha}$ and 
	$
	\sigma_1=(n-\alpha)q-(n+\alpha),\quad \sigma_2=(n-\alpha)r-(n+\alpha).
	$
	If 
	$$(vi)\ \lim_{|x|\to \infty}|x|^{n-\alpha}u(x)=+\infty,\quad  \lim_{|x|\to \infty}|x|^{n-\alpha}v(x)=+\infty,$$ then $u(x)$ and $v(x)$ are radially symmetric about the origin.
\end{Thm}

\begin{Rk}
In \cite{DZ2015},
by showing (\ref{sys:fes3}) is equivalent to some integral system, Dou and Zhou have proved the radially symmetric result as in Theorem \ref{thm:henon}, under the assumption that $u\in L_{loc}^a(\mathbb{R}^n)$, $v\in L_{loc}^b (\mathbb{R}^n)$, where
$a=n(r-1)/\alpha$ and $b=n(q-1)/\alpha$.
\end{Rk}

The paper is organized as follows. We devote Section \ref{sec:pre} to introduce some preliminaries results, including the decay at infinity principle and the narrow region principle for the coupled fractional Laplacian system. We prove Theorem \ref{thm:main} in Section \ref{sec:pfthmmain}. Finally the proofs of Theorem \ref{thm:rev}, Corollary \ref{col:ex1}, \ref{col:ex2} and Theorem \ref{thm:henon} will be presented in Section \ref{sec:cor}. Note that in the following, $c$ and $C$ will be constants which can be different from line to line.

\section{Preliminary}\label{sec:pre}

In what follows, we shall use the method of moving planes.
Choose any direction to be the $x_1$ direction. For each $\lambda\in \mathbb{R}^n$, we write $x=(x_1,x')$ with $x'=(x_2,\dots, x_n)\in \mathbb{R}^{n-1}$ and define
$\Sigma_\lambda:=\{x\in \mathbb{R}^n|x_1<\lambda\}$, $T_\lambda:=\partial \Sigma_\lambda=\{x\in\mathbb{R}^n|x_1=\lambda\}$. For each point $x=(x_1,x')\in \Sigma_\lambda$, let $x^\lambda=(2\lambda-x_1,x')$ be the reflected point with respect to the hyperplane $T_\lambda$. 
Define the reflected functions by $u_\lambda(x)=u(x^\lambda)$, $v_\lambda(x)=v(x^\lambda)$ and introduce functions
$$
U_\lambda (x)=u_\lambda(x)-u(x), \quad V_\lambda (x)=v_\lambda(x)-v(x).
$$

According to the fact that for all $f\in C_{loc}^{1,1}(\mathbb{R}^n)\cap L_\alpha$,
$$((-\Delta)^{\frac{\alpha}{2}}f_\lambda )(x)=((-\Delta)^{\frac{\alpha}{2}}f)(x^\lambda),
$$
it follows that for $x\in \Sigma_\lambda$
\begin{eqnarray}
\label{eq:bigu}
(-\Delta)^{\frac{\alpha}{2}} U_\lambda(x)
 &=& ((-\Delta)^{\frac{\alpha}{2}}  u)(x^\lambda)-((-\Delta)^{\frac{\alpha}{2}}u)(x)\\
\nonumber &=& f(u(x^\lambda),v(x^\lambda))-f(u(x),v(x))\\
\nonumber &=& \frac{\partial f}{\partial u}(\xi_1(x,\lambda),v(x))U_\lambda(x)+\frac{\partial f}{\partial v}(u_\lambda(x),\eta_1(x,\lambda))V_\lambda(x),
\end{eqnarray}
and 
\begin{equation}
\label{eq:bigv}
(-\Delta)^{\frac{\alpha}{2}} V_\lambda(x) =
\frac{\partial g}{\partial u}(\xi_2(x,\lambda),v_\lambda(x))U_\lambda(x)+\frac{\partial g}{\partial v}(u(x),\eta_2(x,\lambda))V_{\lambda}(x),
\end{equation}
where $\xi_i(x,\lambda)$ is between $u_\lambda(x)$ and $
u(x)$, $\eta_i(x,\lambda)$ is between $v_\lambda(x)$ and $
v(x)$, $i=1,2$.

We define 
$$
\Sigma_\lambda^{U-}:=\{x\in\Sigma_\lambda |U_{\lambda}(x)<0\}, \quad 
\Sigma_\lambda^{V-}:=\{x\in\Sigma_\lambda |V_{\lambda}(x)<0\}.
$$

Without the classic maximum principle, we shall introduce the following two principles, which will play an important role in the proof of Theorem \ref{thm:main}.

\begin{Pro}[Decay at Infinity Principle]\label{thm:decay}
	Let $(u,v)\in \left(L_{\alpha}\cap C^{1,1}_{loc}(\mathbb{R}^n)\right)^2$ is a positive solution of system (\ref{sys:fes}) with $f,g\in C^1([0,+\infty)\times [0,+\infty),\mathbb{R})$.
Under assumptions $(i)-(v)$,
there exists a constant $R_0>0$, for all system (\ref{eq:bigu})(\ref{eq:bigv}) with $\lambda\leq 0$, 
\begin{enumerate}
	\item[(a)] if there is $x^*\in \Sigma_\lambda$, $|x^*|>R_0$ such that $U_\lambda(x^*)=\min_{x\in\overline{\Sigma_{\lambda}}} U_{\lambda}(x)<0$, then 
	$$
	V_{\lambda}(x^*)<2U_{\lambda}(x^*)<0;
	$$
	\item[(b)]  if there is $y^*\in \Sigma_\lambda$, $|y^*|>R_0$ such that $V_\lambda(y^*)=\min_{y\in \overline{\Sigma_{\lambda}}} V_{\lambda}(y)<0$, then 
	$$U_{\lambda}(y^*)<2V_{\lambda}(y^*)<0.$$
\end{enumerate}
\end{Pro}

\begin{proof}
Let us assume initially that there is some $x^*\in \Sigma_{\lambda}$ satisfying $$U_\lambda(x^*)=\min_{x\in \overline{\Sigma_\lambda}}U_\lambda(x)<0.$$

Using the anti-symmetry property of $U_\lambda(x)$, we compute
\begin{eqnarray*}
(-\Delta)^{\frac{\alpha}{2}}U_\lambda(x^*) &=& C_{n,\alpha}P.V. \int_{\mathbb{R}^n}\frac{U_\lambda(x^*)-U_\lambda(y)}{|x^*-y|^{n+\alpha}}dy\\
&=&  C_{n,\alpha}P.V.\left( \int_{\Sigma_\lambda}\frac{U_\lambda(x^*)-U_\lambda(y)}{|x^*-y|^{n+\alpha}}dy +\int_{\mathbb{R}^n\backslash\Sigma_\lambda}\frac{U_\lambda(x^*)-U_\lambda(y)}{|x^*-y|^{n+\alpha}}dy\right)\\
&=&  C_{n,\alpha}P.V.\left( \int_{\Sigma_\lambda}\frac{U_\lambda(x^*)-U_\lambda(y)}{|x^*-y|^{n+\alpha}}dy +\int_{\Sigma_\lambda}\frac{U_\lambda(x^*)-U_\lambda(y^\lambda)}{|x^*-y^\lambda|^{n+\alpha}}dy\right)\\
&=&  C_{n,\alpha}P.V.\left( \int_{\Sigma_\lambda}\frac{U_\lambda(x^*)-U_\lambda(y)}{|x^*-y|^{n+\alpha}}dy +\int_{\Sigma_\lambda}\frac{U_\lambda(x^*)+U_\lambda(y)}{|x^*-y^\lambda|^{n+\alpha}}dy\right).
\end{eqnarray*}
Since $x^*, y\in \Sigma_\lambda$, $|x^*-y|\leq |x^*-y^\lambda|$, and $x^*$ is the minimum point of $U_\lambda(x)$ in $\Sigma_\lambda$, it follows that
\begin{eqnarray*}
(-\Delta)^{\frac{\alpha}{2}}U_\lambda(x^*) &\leq&  C_{n,\alpha}\left( \int_{\Sigma_\lambda}\frac{U_\lambda(x^*)-U_\lambda(y)}{|x^*-y^\lambda|^{n+\alpha}}dy +\int_{\Sigma_\lambda}\frac{U_\lambda(x^*)+U_\lambda(y)}{|x^*-y^\lambda|^{n+\alpha}}dy\right)\\
&=& 2C_{n,\alpha} \int_{\Sigma_\lambda}\frac{1}{|x^*-y^\lambda|^{n+\alpha}}dy\cdot U_\lambda(x^*).
\end{eqnarray*}
Combining the above estimate with equation (\ref{eq:bigu}), we obtain
\begin{equation}\label{eq:estu}
\frac{\partial f}{\partial v}(u_\lambda(x^*),\eta_1(x^*,\lambda))V_\lambda(x^*) \leq 
b_1(x^*,\lambda)U_\lambda(x^*),
\end{equation}
where 
$$
b_1(x^*,\lambda)=
2C_{n,\alpha} \int_{\Sigma_\lambda}\frac{1}{|x^*-y^\lambda|^{n+\alpha}}dy
-\frac{\partial f}{\partial u}(\xi_1(x^*,\lambda),v(x^*)).
$$

We claim that  $b_1(x^*,\lambda)>0$, for sufficiently large $|x^*|$,.
In fact, for each fixed $\lambda\leq 0$, since $x^*\in \Sigma_\lambda$, we have $B_{|x^*|}(x^1)\subset \mathbb{R}^n\backslash \Sigma_{\lambda}$, where $x^1=(3|x^*|+x^*_1,(x^*)')$, and it follows that 
\begin{eqnarray*}
	\int_{\Sigma_\lambda}\frac{1}{|x^*-y^\lambda|^{n+\alpha}}dy &=& \int_{\mathbb{R}^n\backslash \Sigma_\lambda}\frac{1}{|x^*-y|^{n+\alpha}}dy\\
	&\geq& \int_{B_{|x^*|}(x^1)}\frac{1}{|x^*-y|^{n+\alpha}}dy\\
	&\geq& \int_{B_{|x^*|}(x^1)}\frac{1}{4^{n+\alpha}|x^*|^{n+\alpha}}dy\\
	&=& \frac{\omega_n}{4^{n+\alpha}|x^*|^{\alpha}},
\end{eqnarray*}
where $\omega_n$ is the volume of the unit ball in $\mathbb{R}^n$.
Therefore, 
\begin{equation}
\label{est:b1}
b_{1}(x^*,\lambda)\geq \frac{2C_{n,\alpha}\omega_n}{4^{n+\alpha}|x^*|^{\alpha}}
-\frac{\partial f}{\partial u}(\xi_1(x^*,\lambda),v(x^*)).
\end{equation}
Owning to $(ii)$, we choose $\delta>0$ such that $\frac{\partial f}{\partial u}(u,v)\lesssim u^{p-1}v^q$ provided that $u+v<\delta$ and $u,v>0$. For this particular $\delta$, because of $(i)$, there exists $R_1>0$ such that $0<u(x^*)<1/|x^*|^a$, $0<v(x^*)<1/|x^*|^b$ and $u(x^*)+v(x^*)<\delta$, when $|x^*|>R_1$. Moreover, as $x^*\in \Sigma_{\lambda}^{U-}$, $0<u_\lambda(x^*)<\xi_1(x^*,\lambda)< u(x^*)$. Hence, for all $|x^*|>R_1$, we have
\begin{eqnarray}
\label{est:fu}\frac{\partial f}{\partial u}(\xi_1(x^*,\lambda),v(x^*))
&\leq&
 \xi_1(x^*,\lambda)^{p-1}v(x^*)^q\\
\nonumber &<& u(x^*)^{p-1}v(x^*)^q\\
\nonumber &\leq & \frac{1}{|x^*|^{a(p-1)+qb}}.
\end{eqnarray}
Putting (\ref{est:fu}) into (\ref{est:b1}) and using assumption $(v)$, we can choose $R_2>R_1$, such that for $|x^*|>R_2$, 
\begin{equation}\label{est:b11}
b_{1}(x^*,\lambda)\geq c|x^*|^{-\alpha}>0.
\end{equation}

Thus, combining (\ref{est:b11}) with assumption $(iv)$, we derive from (\ref{eq:estu}) that
\begin{equation*}
V_{\lambda}(x^*)\leq \frac{b_1(x^*,\lambda)}{\frac{\partial f}{\partial v}(u_\lambda(x^*),\eta_1(x^*,\lambda))}U_\lambda(x^*)<0,
\end{equation*}
which implies $x^*\in \Sigma_{\lambda}^{U-}\cap  \Sigma_{\lambda}^{V-}$.

Conclusion (a) will be proved if we can show that $\frac{b_1(x^*,\lambda)}{\frac{\partial f}{\partial v}(u_\lambda(x^*),\eta_1(x^*,\lambda))}>2$ for sufficiently large $|x^*|$.
Now using assumption $(iii)$, there is $\delta'>0$, such that $\frac{\partial f}{\partial v}(u,v)\lesssim u^pv^{q-1}$ for $u,v>0$, $u+v<\delta'$. Hence, we can choose $R_3>R_2$ so that for all $|x^*|>R_3$ we have
$$
0<u(x^*)<1/|x^*|^a,\quad 0<v(x^*)<1/|x^*|^b,\quad u(x^*)+v(x^*)<\delta’.
$$
Since we have already known that $x^*\in \Sigma_\lambda^{U-}\cap \Sigma_\lambda^{V-}$, and so
$0<u_{\lambda}(x^*)<\xi_1(x^*,\lambda)<u(x^*)$,
$0<v_{\lambda}(x^*)<\eta_1(x^*,\lambda)<v(x^*)$, it follows that
for $|x^*|>R_3$,
\begin{eqnarray}
\label{est:fv}\frac{\partial f}{\partial v}(u_\lambda(x^*),\eta_1(x^*,\lambda))
&\leq&
u_\lambda(x^*)^p\eta_1(x^*,\lambda)^{q-1}\\
\nonumber &<& u(x^*)^{p}v(x^*)^{q-1}\\
\nonumber &\leq & \frac{1}{|x^*|^{ap+qb-b}}.
\end{eqnarray}
In view of (\ref{est:b1}) and (\ref{est:fv}), we obtain for $|x^*|>R_3$
\begin{equation*}
\frac{b_1(x^*,\lambda)}{\frac{\partial f}{\partial v}(u_\lambda(x^*),\eta_1(x^*,\lambda))}>c|x^*|^{ap+b(q-1)-\alpha}\to \infty, \textrm{as}\ |x^*|\to \infty.
\end{equation*}
Therefore, there is $R_0>R_3$, for all $|x^*|>R_0$, we have
\begin{equation}
\label{eq:b1}
\frac{b_1(x^*,\lambda)}{\frac{\partial f}{\partial v}(u_\lambda(x^*),\eta_1(x^*,\lambda))}>2,
\end{equation}
which implies conclusion (a).

By the same token, one can give the proof of (b), which will not be included here.

\end{proof}

\begin{Pro}[Narrow Region Principle]
	\label{thm:nrp}
	Let $(u,v)\in \left(L_{\alpha}\cap C^{1,1}_{loc}(\mathbb{R}^n)\right)^2$ is a positive solution of system (\ref{sys:fes}) with $f,g\in C^1([0,+\infty)\times [0,+\infty),\mathbb{R})$. Assume
	$$
	(i')\quad \lim_{|x|\to \infty}u(x)=0,\quad \lim_{|x|\to \infty}v(x)=0
	$$
	and $(iv)$. Then there exists $l_0>0$ such that for all $0<l\leq l_0$ and all system (\ref{eq:bigu}) and (\ref{eq:bigv})
\begin{itemize}
	\item[(c)]  if there is $x^*\in \Omega_{\lambda,l}:=\{x\in \Sigma_{\lambda}|\lambda-l<x_1<\lambda\}$ satisfying $U_{\lambda}(x^*)=\min_{x\in \overline{\Sigma_{\lambda}}}U_{\lambda}(x)<0$, then 
	$$
	V_{\lambda}(x^*)<2U_{\lambda}(x^*)<0;
	$$
	\item[(d)] if there is $y^*\in \Omega_{\lambda,l}=\{x\in \Sigma_{\lambda}|\lambda-l<x_1<\lambda\}$ satisfying $V_{\lambda}(y^*)=\min_{y\in \overline{\Sigma_{\lambda}}}V_{\lambda}(y)<0$, then 
	$$
	U_{\lambda}(y^*)<2V_{\lambda}(y^*)<0.
	$$
\end{itemize}

\end{Pro}
\begin{proof}
Without loss of generality, let
$$
x^*\in \Omega_{\lambda,l},\ \textrm{and}\  U_{\lambda}(x^*)=\min_{x\in \overline{\Sigma_\lambda}}U_{\lambda}(x^*)<0,
$$
for some $l>0$.

The same computation as in the proof of Proposition \ref{thm:decay} gives us that
\begin{equation*}
	(-\Delta)^{\frac{\alpha}{2}}U_\lambda(x^*)
	\leq C_{n,\alpha} \int_{\Sigma_\lambda}\frac{2U_\lambda(x^*)}{|x^*-y^\lambda|^{n+\alpha}}dy
	< 0
\end{equation*}
and 
\begin{equation}\label{eq:estu2}
\frac{\partial f}{\partial v}(u_\lambda(x^*),\eta_1(x^*,\lambda))V_\lambda(x^*) \leq 
b_1(x^*,\lambda)U_\lambda(x^*),
\end{equation}
where 
$$
b_1(x^*,\lambda)=
2C_{n,\alpha} \int_{\Sigma_\lambda}\frac{1}{|x^*-y^\lambda|^{n+\alpha}}dy
-\frac{\partial f}{\partial u}(\xi_1(x^*,\lambda),v(x^*)).
$$

We claim that for sufficiently small $l$, $b_1(x^*,\lambda)>0$. On the one hand, following \cite{CLL14}, let $D=\{y\ |\ l<y_1-x_1^*<1,|y'-(x^*)'|<1 \}$, $s=y_1-x_1^*$, $\tau\cdot (\varphi_1,\dots, \varphi_{n-1})=y'-(x^*)'$, $(\varphi_1,\dots, \varphi_{n-1})\in S^{n-2}(1)$ and $\omega_{n-2}=|B_1(0)|$ in $\mathbb{R}^{n-2}$.
A direct computation shows
\begin{eqnarray*}
\int_{\Sigma_\lambda}\frac{1}{|x^*-y^\lambda|^{n+\alpha}}dy 
&=& \int_{\mathbb{R}^n\backslash\Sigma_\lambda}\frac{1}{|x^*-y|^{n+\alpha}}dy\\
&\geq& \int_D \frac{1}{|x^*-y|^{n+\alpha}}dy\\
&=& \int_l^1 \int_0^1 \frac{\omega_{n-2}\tau^{n-2}d\tau}{(s^2+\tau^2)^{\frac{n+\alpha}{2}}}ds\\
&=& \int_l^1 \int_{0}^{\frac{1}{s}}\frac{\omega_{n-2}(st)^{n-2}sdt}{s^{n+\alpha}(1+t^2)^{\frac{n+\alpha}{2}}}ds\quad (\textrm{Let}\ \tau=st)\\
&=& \int_l^1 \frac{1}{s^{1+\alpha}} \int_{0}^{\frac{1}{s}}\frac{\omega_{n-2}t^{n-2} dt}{(1+t^2)^{\frac{n+\alpha}{2}}}ds\\
&\geq&  \int_l^1 \frac{1}{s^{1+\alpha}} \int_{0}^{1}\frac{\omega_{n-2}t^{n-2} dt}{(1+t^2)^{\frac{n+\alpha}{2}}}ds\\
&\geq& C\int_l^1 \frac{1}{s^{1+\alpha}}ds\to +\infty, (\textrm{as}\ l\to 0^+). 
\end{eqnarray*}
On the other hand, thanks to $(i')$, the solution $u(x)$ and $v(x)$ are bounded functions on $\mathbb{R}^n$, which means $\xi(x,\lambda)$ is also bounded. On account of $f\in C^1([0,+\infty)\times [0,+\infty),\mathbb{R})$, we have some $c>0$ such that
$$
\left|\frac{\partial f}{\partial u}(\xi_1(x^*,\lambda),v(x^*))\right|<c, \textrm{for all}\ \lambda.
$$
Hence there is $l_1>0$ such that for all $0<l\leq l_1$ if $x^*\in \Omega_{\lambda,l}$ we have
$
b_1(x^*,\lambda)>0.
$ 
Moreover, we actually have
\begin{equation}
\label{lim:b1}
\lim_{l\to 0^+}
b_1(x^*,\lambda)=+\infty.
\end{equation}

Combining the above estimate and assumption $(iv)$,  we derive from (\ref{eq:estu2}) that
for all $0<l<l_1$,
\begin{equation}
\label{eq:negv}
V_{\lambda}(x^*)\leq \frac{b_1(x^*,\lambda)}{\frac{\partial f}{\partial v}(u_\lambda(x^*),\eta_1(x^*,\lambda))}U_{\lambda}(x^*)<0.
\end{equation}

Note that $\frac{\partial f}{\partial v}(u_\lambda(x^*),\eta_1(x^*,\lambda))$ is uniformly bounded with respect ot $\lambda$, i.e. there is $c>0$ such that 
\begin{equation}
\label{est:fvbd}
0<\frac{\partial f}{\partial v}(u_\lambda(x^*),\eta_1(x^*,\lambda))<c, \forall \lambda\in \mathbb{R}.
\end{equation}
Using  (\ref{lim:b1}), (\ref{eq:negv}) and (\ref{est:fvbd}), we can choose $0<l_0<l_1$ such that 
for all $0<l<l_0$, 
\begin{equation}
\label{est:b1fvlt2}
\frac{b_1(x^*,\lambda)}{\frac{\partial f}{\partial v}(u_\lambda(x^*),\eta_1(x^*,\lambda))}>2,
\end{equation}
Putting (\ref{est:b1fvlt2}) into (\ref{eq:negv}), we obtain conclusion (c).

The proof of conclusion (d) is similar and we shall omit it.
\end{proof}
 
\section{Proof of Theorem \ref{thm:main}}\label{sec:pfthmmain}

In this section, we prove Theorem \ref{thm:main} by using the direct method of moving planes in the spirit of \cite{CLL14}. We divide our proof in three steps.

\textbf{Step 1. } There exists $\lambda^*<0$ such that
\begin{equation}
\label{eq:posuv}
U_\lambda(x)\geq 0\ \textrm{and}\ V_\lambda(x)\geq 0, \quad \forall x\in \Sigma_\lambda,
\end{equation}
for all $\lambda\leq \lambda^*$.

\begin{proof}[Proof of Step 1]
Choose $\lambda^*<-R_0$, where $R_0$ is given by Proposition \ref{thm:decay}. It follows that $U_{\lambda}(x)\geq 0$ and $V_{\lambda}(x)\geq 0$ in $\Sigma_{\lambda}$ for all $\lambda\leq \lambda^*$. 

Assume for contradiction that there is a $\lambda\leq \lambda^*$ and a point $x^*\in \Sigma_\lambda$ such that $U_{\lambda}(x^*)<0$. Without loss of generality, we assume 
$$
U_{\lambda}(x^*)=\min_{x\in\overline{\Sigma_{\lambda}}} U_{\lambda}(x)<0.
$$
Since $\lambda\leq \lambda^*<-R_0$, we know that $|x^*|>R_0$. An immediate consequence of Proposition \ref{thm:decay} is
\begin{equation}
\label{est:v1}
V_{\lambda}(x^*)< 2U_{\lambda}(x^*)<0.
\end{equation}
Since assumption $(i')$, which implies $\lim_{|x|\to \infty}V_\lambda(x)=0$, and $V_\lambda(x)=0$ for $x\in T_\lambda$, there exists a point $y^*\in \Sigma_{\lambda}$ such that
\begin{equation*}
V_{\lambda}(y^*)=\min_{y\in \overline{\Sigma_\lambda}}V_\lambda(y)<0.
\end{equation*}
In view of $|y^*|>R_0$, it follows from Proposition \ref{thm:decay} that
\begin{equation}
\label{ineq:s1b}
U_{\lambda}(y^*)<2V_\lambda(y^*)<0.
\end{equation}
Combining (\ref{est:v1}) and (\ref{ineq:s1b}), we obtain
$$
V_{\lambda}(x^*)<2U_{\lambda}(x^*)\leq 2U_{\lambda}(y^*)\leq 4V_{\lambda}(y^*)< 4V_{\lambda}(x^*).
$$
Noticing that $V_{\lambda}(x^*)<0$, we get
$
1>4,
$
which is a contradiction.

\end{proof}

We now move the plane $T_\lambda$ to the right as long as (\ref{eq:posuv}) holds to its limiting position.
Define
\begin{equation*}
\lambda_0=\sup \{\lambda\leq 0\ |\ U_{\mu}(x)\geq 0, V_{\mu}(x)\geq 0, \forall x\in \Sigma_{\mu}, \forall \mu\leq \lambda \}.
\end{equation*}
Step 1 indicates that $\lambda_0>-\infty$. Obviously, $\lambda_0\leq 0$. Since all the functions we consider are continuous with respect to $\lambda$, we know that $U_{\lambda_0}(x)\geq 0$ and $V_{\lambda_0}(x)\geq 0$, for all $x\in \Sigma_{\lambda_0}$.

Before proceeding further, we shall investigate the properties of functions $U_{\lambda_0}(x)$ and $V_{\lambda_0}(x)$.
\begin{Claim}
	\label{claim:a}
	If $U_{\lambda_0}(x)\equiv 0$, then $V_{\lambda_0}(x)\equiv 0$. If $V_{\lambda_0}(x)\equiv 0$, then $U_{\lambda_0}(x)\equiv 0$.
\end{Claim}
\begin{proof}
If $U_{\lambda_0}(x)\equiv 0$, then equation (\ref{eq:bigu}) becomes
\begin{equation*}
0 =(-\Delta)^{\frac{\alpha}{2}}U_{\lambda_0}(x)
=\frac{\partial f}{\partial v}(u_{\lambda_0}(x),\eta_1(x,\lambda_0))V_{\lambda_0}(x).
\end{equation*}
From assumption $(iv)$, we have $V_{\lambda_0}(x)\equiv 0$.
\end{proof}

\begin{Claim}
	\label{claim:b} 
	If $U_{\lambda_0}(x)\not\equiv 0$ or $V_{\lambda_0}(x)\not\equiv 0$, then $U_{\lambda_0}(x)>0$ and $V_{\lambda_0}(x)>0$, for all $x\in \Sigma_{\lambda_0}$. 
\end{Claim}
\begin{proof}
Let us take $U_{\lambda_0}(x) \not\equiv 0$ as an example.
Since we have already know $U_{\lambda_0}(x)\geq 0$, for all $x\in \Sigma_{\lambda_0}$. To show $U_{\lambda_0}(x)>0$, $\forall x\in \Sigma_{\lambda_0}$, we assume for contradiction that there is some point $x^*\in \Sigma_{\lambda_0}$, so that
$$
U_{\lambda_0}(x^*)=0.
$$

At point $x^*$, we compute
\begin{eqnarray}
& & \label{est:minu}
(-\Delta )^{\frac{\alpha}{2}}U_{\lambda_0}(x^*)\\
\nonumber &=& C_{n,\alpha} P.V. \int_{\mathbb{R}^n} \frac{-U_{\lambda_0}(y)}{|x^*-y|^{n+\alpha}}dy\\
\nonumber &=&  C_{n,\alpha}P.V. \left( \int_{\Sigma_{\lambda_0}}\frac{-U_{\lambda_0}(y)}{|x^*-y|^{n+\alpha}}dy +\int_{\mathbb{R}^n\backslash\Sigma_{\lambda_0}}\frac{-U_{\lambda_0}(y)}{|x^*-y|^{n+\alpha}}dy\right)\\
\nonumber &=&  C_{n,\alpha}P.V.\left( \int_{\Sigma_{\lambda_0}}\frac{-U_{\lambda_0}(y)}{|x^*-y|^{n+\alpha}}dy +\int_{\Sigma_{\lambda_0}}\frac{-U_{\lambda_0}(y^{\lambda_0})}{|x^*-y^{\lambda_0}|^{n+\alpha}}dy\right)\\
\nonumber &=& C_{n,\alpha}P.V.\left( \int_{\Sigma_{\lambda_0}}\frac{-U_{\lambda_0}(y)}{|x^*-y|^{n+\alpha}}dy +\int_{\Sigma_{\lambda_0}}\frac{U_{\lambda_0}(y)}{|x^*-y^{\lambda_0}|^{n+\alpha}}dy\right)\\
\nonumber &=&  C_{n,\alpha}P.V.\left( \int_{\Sigma_{\lambda_0}}U_{\lambda_0}(y)\left(\frac{1}{|x^*-y^{\lambda_0}|^{n+\alpha}}-\frac{1}{|x^*-y|^{n+\alpha}}\right)dy\right).
\end{eqnarray}
From the fact that $|x^*-y^{\lambda_0}|>|x^*-y|$, $y\in \Sigma_{\lambda_0}$, $U_{\lambda_0}(y)\geq 0$ and $U_{\lambda_0}(y)\not\equiv 0$, it turns out that
\begin{equation*}
(-\Delta )^{\frac{\alpha}{2}}U_{\lambda_0}(x^*)< 0.
\end{equation*}

On the other hand at $x^*$, the right hand side of equation (\ref{eq:bigu}) equals to
$$
\frac{\partial f}{\partial v}(u_{\lambda_0}(x^*),\eta_1(x^*,\lambda_0))V_{\lambda_0}(x^*)\geq 0,
$$
which leads to a contradiction.
Consequently, we obtain $U_{\lambda_0}(x)>0$, for all $x\in \Sigma_{\lambda_0}$.
In such case, Claim \ref{claim:a} suggests us that  $V_{\lambda_0}(x)\not\equiv 0$. Using a similar argument, one can show that $V_{\lambda_0}(x)>0$, for all $x\in \Sigma_{\lambda_0}$.

\end{proof}

We proceed to prove Theorem \ref{thm:main}.

\textbf{Step 2.}
If $\lambda_0<0$, then 
$$
U_{\lambda_0}(x)\equiv 0, \quad V_{\lambda_0}(x)\equiv 0, \quad \forall x\in \Sigma_{\lambda_0}.
$$

\begin{proof}[Proof of Step 2.]
By Claim \ref{claim:a} and \ref{claim:b}, it suffices to exclude the situation when both $U_{\lambda_0}$ and $V_{\lambda_0}$ are strictly positive in $\Sigma_{\lambda_0}$. Let us suppose this is the case, i.e. 
\begin{equation}
\label{eq:pouv}
U_{\lambda_0}(x)>0, V_{\lambda_0}(x)>0, \forall x\in \Sigma_{\lambda_0}.
\end{equation} 

From the definition of $\lambda_0$ ($<0$), there exist sequences $\{\lambda_k\}_{k=1}^\infty$ and $\{x^k\}_{k=1}^\infty$ satisfying
\begin{equation}
\label{def:lamk}
\lambda_0<\lambda_{k+1}<\lambda_k<0, k=1,2,\dots;\quad \lim_{k\to \infty}\lambda_k=\lambda_0;
\end{equation}
$
x^k\in \Sigma_{\lambda_k}$, and either $U_{\lambda_k}(x^k)<0$ or $V_{\lambda_k}(x^k)<0$. Taking $U_{\lambda_k}(x^k)<0$ (up to a subsequence) as an example, we can rename $x^k$ to be the minimum points, i.e.
\begin{equation}
\label{eq:uxk}
U_{\lambda_k}(x^k)=\min_{x\in \overline{\Sigma_k}} U_{\lambda_k}(x)<0,\quad k=1,2,\dots.
\end{equation}

There are two possible cases.

\textbf{Case 1.}
The sequence $\{x^k\}_{k=1}^\infty$ contains a bounded subsequence. 

Without loss of generality, we assume 
\begin{equation}
\label{eq:limxk}
\lim_{k\to \infty}x^k=x^*.
\end{equation}
From (\ref{eq:uxk}) and (\ref{eq:pouv}), it is easy to see
$$
x^*\in \cap_{k=1}^{+\infty}\Sigma_{\lambda_0}=\overline{\Sigma_{\lambda_0}}, \quad 
U_{\lambda_0}(x^*)=0.
$$
Since $U_{\lambda_0}(x)>0$ for $x\in \Sigma_{\lambda_0}$, $x^*$ has to lie on the boundary of $\Sigma_{\lambda_0}$, i.e. $x^*\in T_{\lambda_0}$.
On account of (\ref{def:lamk}) and (\ref{eq:limxk}), for $l_0>0$ defined by Proposition \ref{thm:nrp}, we can choose $K_1>0$ so that if $k>K_1$, $\lambda_0<\lambda_k<\lambda_0+l_0/2$ and 
$$
x^k\in \Omega_{\lambda_k+l_0/2,l_0}.
$$
According to Proposition \ref{thm:nrp}, we have
\begin{equation}
\label{ineq:vxk2}
V_{\lambda_k}(x^k)< 2U_{\lambda_k}(x^k)<0.
\end{equation}
It follows that there is
$y^k\in\Sigma_{\lambda_k}$, such that
$$
V_{\lambda_k}(y^k)=\min_{y\in \overline{\Sigma_{\lambda_k}}}V_{\lambda_k}(y)<0,\quad  k=K_1+1,K_1+2,\dots.
$$
For sequence $\{y^k\}_{K_1+1}^\infty$, there are also two possible cases. 

\textbf{Case 1.1}
$\{y^k\}_{K_1+1}^\infty$ has a bounded subsequence.

In this case, we also denote the convergent subsequence as $y^k$, i.e. $\lim_{k\to \infty}y^k=y^*$. It follows that
$V_{\lambda_0}(y^*)=0$ and so $y^*\in T_{\lambda_0}$. Hence, there is $K>K_1$, for $k>K$, $y^k\in \Omega_{\lambda_k+l_0/2,l_0}$. By using Proposition \ref{thm:nrp}, 
we have 
\begin{equation}
\label{ineq:uyk2}
U_{\lambda_k}(y^k)< 2V_{\lambda_k}(y^k)<0.
\end{equation}

Combining (\ref{ineq:uyk2}) with (\ref{ineq:vxk2}), we obtain for $k>K$
$$
U_{\lambda_k}(y^k)<2V_{\lambda_k}(y^k)\leq 2V_{\lambda_k}(x^k)< 4U_{\lambda_k}(x^k)\leq 4U_{\lambda_k}(y^k),
$$ 
which leads to a contradiction since $U_{\lambda_k}(y^k)<0$.

\textbf{Case 1.2} $\lim_{k\to \infty}|y^k|=\infty$.

In this case, there is $K>K_1$, such that for all $k>K$, $|y^k|>R_0$.
It follows from Proposition \ref{thm:decay} that
\begin{equation}
\label{ineq:uyk3}
U_{\lambda_k}(y^k)< 2V_{\lambda_k}(y^k)<0,\quad k>K.
\end{equation}

Combining (\ref{ineq:uyk3}) with (\ref{ineq:vxk2}), we find for $k>K$
$$
U_{\lambda_k}(y^k)<2V_{\lambda_k}(y^k)\leq 2V_{\lambda_k}(x^k)< 4U_{\lambda_k}(x^k)\leq 4U_{\lambda_k}(y^k), 
$$ 
which leads to a contradiction since $U_{\lambda_k}(y^k)<0$.

\textbf{Case 2.}
$\lim_{k\to \infty}|x^k|=\infty$.

In this case, there is $K_2>0$ such that for $k>K_2$, we have $|x^k|>R_0$.
Because of Proposition \ref{thm:decay},
\begin{equation}
\label{ineq:vxk4}
V_{\lambda_k}(x^k)<2 U_{\lambda_k}(x^k)<0, \forall k>K_2.
\end{equation}
Hence, for each $k>K_2$ there is $y^k\in\Sigma_{\lambda_k}$ satisfies
$$
V_{\lambda_k}(y^k)=\min_{y\in \overline{\Sigma_k}}V_{\lambda_k}(y)<0,\quad \forall k>K_2.
$$

We consider the following two possible cases.

\textbf{Case 2.1} $\{y^k\}$ has a bounded subsequence.

By a similar argument as in Case 1.2, we will find a contradiction in this case.

\textbf{Case 2.2}
$\lim_{k\to \infty}|y^k|=\infty$.

We can choose $K>K_2$, such that for $k>K$, $|y^k|>R_0$. According to Proposition \ref{thm:decay}, we obtain
\begin{equation}
\label{ineq:uyk4}
U_{\lambda_k}(y^k)<2V_{\lambda_k}(y^k)<0,\quad k>K.
\end{equation}
Therefore, (\ref{ineq:uyk4}) and (\ref{ineq:vxk4}) tell us that when $k>K$
$$
U_{\lambda_k}(y^k)<2V_{\lambda_k}(y^k)\leq 2V_{\lambda_k}(x^k)< 4U_{\lambda_k}(x^k)\leq 4U_{\lambda_k}(y^k),
$$ 
which leads to a contradiction since $U_{\lambda_k}(y^k)<0$.

To conclude, if $\lambda_0<0$, we have
$$
U_{\lambda_0}\equiv 0, V_{\lambda_0}(x)\equiv 0, \forall x\in \Sigma_{\lambda_0}.
$$

\end{proof}

\textbf{Step 3.}
The solution $(u,v)$ is radially symmetric with respect to some point in $\mathbb{R}^n$.

\begin{proof}[Proof of Step 3.]
	
First, we show that $(u,v)$ is symmetric with respect to some hyperplane $\{x\in \mathbb{R}^n|x_1=c\}$.

Clearly, the consequence of Step 2 says that if $\lambda_0<0$, then $(u,v)$ is symmetric with respect to hyperplane $\{x\in \mathbb{R}^n|x_1=\lambda_0\}$.

If $\lambda_0=0$, we move the plane $T_\lambda$ from the $+\infty$ to the left.
Define 
$$
\lambda_0'=\inf \{\lambda\geq 0\ |\ U_{\mu}(x)\leq 0, V_{\mu}(x)\leq 0, \forall x\in \Sigma_{\mu}, \forall \mu\geq \lambda \}.
$$
If $\lambda_0'>0$,
an argument similar to the one used in Step 2 shows that
$U_{\lambda_0'}\equiv 0$ and $V_{\lambda_0'}\equiv 0$, which implies $(u,v)$ is symmetric with respect to hyperplane $\{x\in \mathbb{R}^n|x_1=\lambda_0'\}$.
If $\lambda_0'=0=\lambda_0$, then $U_{0}\geq 0$, $V_0\geq 0$, $\forall x\in \Sigma_0$ and $U_{0}\leq 0$, $V_0\leq 0$, $\forall x\in \Sigma_0$. So we must have $U_{0}\equiv 0$ and $V_{0}\equiv 0$, i.e. $(u,v)$ is symmetric with respect to hyperplane $\{x\in \mathbb{R}^n|x_1=0\}$.

Since the $x_1$ direction can chosen arbitarily, we have actually show that $(u,v)$ are radially symmetric about some point in $\mathbb{R}^n$.

\end{proof}

\section{Proof of Corollaries}\label{sec:cor}

\begin{proof}[Proof of Corollary \ref{col:ex1}]
Set $f(u,v)=\sum_{i=1}^Nu^{p_i}v^{q_i}$ and $g(u,v)=\sum_{i=1}^Nu^{r_i}v^{s_i}$. Choosing 
$0<p'<p$, $0<q'<q$, $0<r'<r$, $0<s'<s$ such that 
$$
\alpha<\min\{ap'+bq'-a, ap'+bq'-b, ar'+bs'-a, ar'+bs'-b\},
$$
we can easily check that all the assumptions in Theorem \ref{thm:main} are satisfied.
\end{proof}

\begin{proof}[Proof of Theorem \ref{thm:rev}]
The proof of this result is quite similar to the one of Theorem \ref{thm:main}. Exam carefully the proof of Proposition \ref{thm:decay}, we see that if the assumptions $(ii)$ $(iii)$ and $(v)$ are replaced by $(ii')$ $(iii')$ and $(v')$ respectively, (\ref{est:fu}) is replaced by 
\begin{equation*}
\frac{\partial f}{\partial u}(\xi_1(x^*,\lambda),v(x^*))
\leq
\xi_1(x^*,\lambda)^{p-1}
<u(x^*)^{p-1}
\leq \frac{1}{|x^*|^{a(p-1)}},
\end{equation*}
and (\ref{est:fv}) is replaced by
\begin{equation*}
\frac{\partial f}{\partial v}(u_\lambda(x^*),\eta_1(x^*,\lambda))
\leq
\eta_1(x^*,\lambda)^{q-1}
< v(x^*)^{q-1}
\leq \frac{1}{|x^*|^{b(q-1)}},
\end{equation*}
then inequalities (\ref{est:b1}) and (\ref{eq:b1}) also hold. So does the conclusion of Proposition \ref{thm:decay}. The conclusion of Proposition \ref{thm:nrp} holds since it does not assume $(ii)$, $(iii)$ or $(v)$.
Therefore, it is straightforward to show Theorem \ref{thm:rev}. The detail pf proof will be omitted.
\end{proof}

\begin{proof}[Proof of Corollary \ref{col:ex2}]
 
Let $f(u,v)=v^{(n+\alpha)/(n-\alpha)}-\kappa_1 u^p$, $g(u,v)=u^{(n+\alpha)/(n-\alpha)}-\kappa_2 v^s$. 
Choosing any $p',s'>n/(n-\alpha)$, $q',r'\in(n/(n-\alpha), (n+\alpha)/{n-\alpha})$, we can see that all the assumptions in Theorem \ref{thm:rev} are satisfied. Hence the conclusion is verified by applying Theorem \ref{thm:rev}.
\end{proof}

\begin{proof}[Proof of Theorem \ref{thm:henon}]
	Since no decay condition on the solution is assumed, we make a Kelvin transform. Let
	$$
	\bar u(x)=\frac{1}{|x|^{n-\alpha}}u(\frac{x}{|x|^2}),\quad \bar v(x)=\frac{1}{|x|^{n-\alpha}}v(\frac{x}{|x|^2}).
	$$
Clearly,
	$$
	\bar u(x)\sim |x|^{-(n-\alpha)}, \quad \bar v(x)\sim |x|^{-(n-\alpha)},\quad \textrm{as}\ |x|\to \infty,
	$$
	and $\lim_{x\to 0} \bar u(x)=\lim_{x\to 0} \bar v(x)=+\infty$.
	It is well known that 
\begin{equation}\label{sys:baruv}
\left\{
\begin{array}{ll}
(-\Delta)^{\frac{\alpha}{2}}\bar u(x)=(\bar v(x))^q, & x\in \mathbb{R}^n\backslash\{0\},
\\
(-\Delta)^{\frac{\alpha}{2}}\bar v(x)=(\bar u(x))^r, & x\in \mathbb{R}^n\backslash\{0\}.
\end{array}
\right.
\end{equation}
Our task now is to show that $\bar u$ and $\bar v$ are symmetric about the origin.

We use the same notation as in Section \ref{sec:pre}. 
Let $\bar u_{\lambda}(x)=\bar u(x^\lambda)$, $\bar v_{\lambda}(x)=\bar v(x^\lambda)$,
$$
U_{\lambda}(x)=\bar u_{\lambda}(x)-\bar u(x), \quad V_{\lambda}(x)=\bar v_{\lambda}(x)-\bar v(x).
$$
For $\lambda<0$, 
we have $\lim_{|x|\to \infty}U_{\lambda}(x)=0$, and $\lim_{x\to 0^\lambda}U_{\lambda}(x)=+\infty$ (by using assumption (vi)), where $0^\lambda=(2\lambda,0,\dots,0)$.
Therefore, if there is some point $x\in \Sigma_{\lambda}$, such that $U_{\lambda}(x)<0$, then $U_{\lambda}(x)$ attains its negative minimum in $\Sigma_{\lambda}$ and the minimum point belongs to $\Sigma_\lambda\backslash \{0^\lambda\}$. Similar results hold for $V_{\lambda}$.

It follows from (\ref{sys:baruv}) and the definition of $U_\lambda$ and $V_\lambda$ that
\begin{equation}\label{sys:biguv}
\left\{
\begin{array}{ll}
(-\Delta)^{\frac{\alpha}{2}}U_{\lambda}(x)=q\eta(x,\lambda)^{q-1}V_\lambda(x), & x\in \Sigma_{\lambda}\backslash\{0^\lambda\},
\\
(-\Delta)^{\frac{\alpha}{2}}U_\lambda(x)=r\xi(x,\lambda)^{r-1}U_{\lambda}(x), & x\in \Sigma_{\lambda}\backslash\{0^\lambda\},
\end{array}
\right.
\end{equation}
where $\xi(x,\lambda)$ is between $\bar u_\lambda(x)$ and $\bar u(x)$; $\eta(x,\lambda)$ is between $\bar v_\lambda(x)$ and $\bar v(x)$.

Using an argument similar to the one used in Section \ref{sec:pre}, we have the following decay at infinity principle and narrow region principle for system (\ref{sys:biguv}).

\begin{Lem}[Decay at Infinity Principle]\label{thm:decay2}
	There exists a constant $R_0>0$, for all system (\ref{sys:biguv}) with $\lambda< 0$, 
	\begin{enumerate}
		\item[(a)] if there is $x^*\in \Sigma_\lambda$, $|x^*|>R_0$ such that $U_\lambda(x^*)=\min_{x\in\overline{\Sigma_{\lambda}}} U_{\lambda}(x)<0$, then 
		$$
		V_{\lambda}(x^*)<2U_{\lambda}(x^*)<0;
		$$
		\item[(b)]  if there is $y^*\in \Sigma_\lambda$, $|y^*|>R_0$ such that $V_\lambda(y^*)=\min_{y\in \overline{\Sigma_{\lambda}}} V_{\lambda}(y)<0$, then 
		$$U_{\lambda}(y^*)<2V_{\lambda}(y^*)<0.$$
	\end{enumerate}
\end{Lem}

\begin{Lem}[Narrow Region Principle]
	\label{thm:nrp2}
	For any fixed $\lambda^*<0$. There exists $l_0>0$ such that for all $0<l\leq l_0$ and all system (\ref{sys:biguv}) with $\lambda\leq \lambda^*<0$
	\begin{itemize}
		\item[(c)]  if there is $x^*\in \Omega_{\lambda,l}:=\{x\in \Sigma_{\lambda}|\lambda-l<x_1<\lambda\}$ satisfying $U_{\lambda}(x^*)=\min_{x\in \overline{\Sigma_{\lambda}}}U_{\lambda}(x)<0$, then 
		$$
		V_{\lambda}(x^*)<2U_{\lambda}(x^*)<0;
		$$
		\item[(d)] if there is $y^*\in \Omega_{\lambda,l}=\{x\in \Sigma_{\lambda}|\lambda-l<x_1<\lambda\}$ satisfying $V_{\lambda}(y^*)=\min_{y\in \overline{\Sigma_{\lambda}}}V_{\lambda}(y)<0$, then 
		$$
		U_{\lambda}(y^*)<2V_{\lambda}(y^*)<0.
		$$
	\end{itemize}
\end{Lem}

The proofs of Lemma \ref{thm:decay2} and \ref{thm:nrp2} are similar to those of Proposition \ref{thm:decay} and \ref{thm:nrp}. So we omit them.

By using Lemma \ref{thm:decay2} instead of Proposition \ref{thm:decay}, Step 1 is the same as in Section \ref{sec:pfthmmain}, that is we can show that there is $\lambda^*<-R_0$ such that
$$
U_{\lambda}(x)\geq 0, \quad V_{\lambda}(x)\geq 0, \forall x\in \Sigma_{\lambda}\backslash\{0^\lambda\},
$$
for all $\lambda\leq\lambda^*$.

Let 
$$
\lambda_0=\sup\{\lambda\leq 0|U_{\mu}(x)\geq 0, V_{\mu}(x)\geq 0, \forall x\in \Sigma_{\lambda}\backslash\{0^\lambda\}, \forall \mu\leq \lambda \}.
$$

\begin{Claim}\label{cla:lam0}
	$\lambda_0=0.$
\end{Claim}
\begin{proof}[Proof of Claim \ref{cla:lam0}]
	If the statement was false, then $\lambda_0<0$. 
	
	We know from continuity that $U_{\lambda_0}(x)\geq 0$ and $V_{\lambda_0}(x)\geq 0$ for all $x\in \Sigma_{\lambda}\backslash\{0^\lambda\}$.
	Since $\lim_{x\to 0^{\lambda_0}}U_{\lambda_0}(x)=+\infty$ and $\lim_{|x|\to \infty}U_{\lambda_0}(x)=0$, $U_{\lambda_0}(x)\not\equiv 0$. 
	
	Actually, we have $U_{\lambda_0}(x)>0$, for all $x\in \Sigma_{\lambda}\backslash\{0^\lambda\}$. If the assertion would not hold, then there is $x^*\in \Sigma_{\lambda}\backslash\{0^\lambda\}$ such that $U_{\lambda_0}(x^*)=0$. By
	using a similar computation as in (\ref{est:minu}), we have
	\begin{equation}
	\label{ineq:negu}
	(-\Delta)^{\frac{\alpha}{2}}U_{\lambda_0}(x^*)<0.
	\end{equation}
	While system (\ref{sys:biguv}) suggests that 
	$$
	(-\Delta)^{\frac{\alpha}{2}}U_{\lambda_0}(x^*)=q\eta(x^*,\lambda_0)^{q-1}V_{\lambda_0}(x^*)\geq 0,
	$$
	which contradicts to (\ref{ineq:negu}). Therefore, 
	$$
	U_{\lambda_0}(x)>0, \forall x\in \Sigma_{\lambda}\backslash\{0^\lambda\}
	$$
	Same argument implies $V_{\lambda_0}(x)>0,  \forall x\in \Sigma_{\lambda}\backslash\{0^\lambda\}$.

From the definition of $\lambda_0$ ($<0$), there exist sequences $\{\lambda_k\}_{k=1}^\infty$ and $\{x^k\}_{k=1}^\infty$ satisfying
\begin{equation*}
\lambda_0<\lambda_{k+1}<\lambda_k<0, k=1,2,\dots;\quad \lim_{k\to \infty}\lambda_k=\lambda_0;
\end{equation*}
$
x^k\in \Sigma_{\lambda_k}$, and either $U_{\lambda_k}(x^k)<0$ or $V_{\lambda_k}(x^k)<0$.
The remainder of the argument is analogous to that of Step 2 in Section \ref{sec:pfthmmain}, that is by using Lemma \ref{thm:decay2} and \ref{thm:nrp2}, we will find that all possible cases lead to contradictions.
\end{proof}

Since $\lambda_0=0$. By moving the planes from $x_1=+\infty$, we derive that $\bar u$ and $\bar v$ are symmetric about the hyperplane $x_1=0$. Since $x_1$ direction can be chosen arbitrarily, we know that $\bar u$ and $\bar v$ are symmetric about the origin, and so do $u$ and $v$. 

\end{proof}

\end{document}